\documentclass[11pt,reqno]{amsart}

 \usepackage{amsfonts,amsmath,amscd,amssymb,amsbsy,amsthm,amstext,amsopn,amsxtra}
 \usepackage{color,fullpage,mathrsfs}
 \usepackage{subfigure}
 \usepackage{verbatim} %to comment out material 
 \usepackage{stmaryrd}
 \usepackage{extarrows}
 \usepackage[all]{xy}
 \usepackage{bm}   %%% used in bold v below
 \usepackage{hyperref}
 \usepackage{tikz}
 \usepackage{tikz-cd}
 \usepackage{enumitem}
 
\usepackage{cancel}
\usepackage{color}
\usepackage{tikz}

 \numberwithin{equation}{section}
\hypersetup{colorlinks,linkcolor=[rgb]{0,0,1},citecolor=[rgb]{1,0,0}}
 
\newtheorem{theorem}{Theorem}[section]
\newtheorem{lemma}[theorem]{Lemma}
\newtheorem{proposition}[theorem]{Proposition}
\newtheorem{corollary}[theorem]{Corollary}

\theoremstyle{definition}

\newtheorem{example}[theorem]{Example}

\theoremstyle{remark}
\newtheorem{remark}[theorem]{Remark}
\newtheorem{remarks}[theorem]{Remarks}

\newcommand{\one}{\ensuremath{(\mathrm{i})}}
\newcommand{\two}{\ensuremath{(\mathrm{ii})}}
\newcommand{\three}{\ensuremath{(\mathrm{iii})}}

\newcommand{\CC}{\ensuremath{\mathbb{C}}} 
\newcommand{\kk}{\ensuremath{\Bbbk}} 
\newcommand{\NN}{\ensuremath{\mathbb{N}}} 
\newcommand{\QQ}{\ensuremath{\mathbb{Q}}} 
\newcommand{\RR}{\ensuremath{\mathbb{R}}}

\newcommand{\ZZ}{\ensuremath{\mathbb{Z}}} 
\newcommand{\GGm}
{\ensuremath{\mathbb{G}_m}}

\newcommand{\Cox}{\operatorname{Cox}}

\newcommand{\Gab}{G_{\operatorname{ab}}}
\newcommand{\git}{\ensuremath{\operatorname{\!/\!\!/\!}}}
\newcommand{\GL}{\operatorname{GL}}
\newcommand{\head}{\operatorname{h}}
\newcommand{\Hom}{\operatorname{Hom}}

\newcommand{\Pic}{\operatorname{Pic}}
\newcommand{\Proj}{\operatorname{Proj}}
\newcommand{\Hilb}{\operatorname{Hilb}}
 
\newcommand{\Rep}{\operatorname{Rep}}
\newcommand{\SL}{\operatorname{SL}}
\newcommand{\Spec}{\operatorname{Spec}}
\newcommand{\tail}{\operatorname{t}}

\newcommand{\sign}{\mathrm{sign}}

\newcommand{\ali}[1]{\textcolor{blue}{#1}}

\begin{document}

\title{The semi-invariant ring as the Cox ring of a GIT quotient}
%quiver variety

\author{Gwyn Bellamy}
\address{School of Mathematics and Statistics, University Place, Glasgow, G12 8QQ, Glasgow, UK.}
\email{gwyn.bellamy@glasgow.ac.uk}
\urladdr{https://www.gla.ac.uk/schools/mathematicsstatistics/staff/gwynbellamy/}

\author{Alastair Craw} 
\address{Department of Mathematical Sciences, 
University of Bath, 
Claverton Down, 
Bath BA2 7AY, 
UK.}
\email{a.craw@bath.ac.uk}
\urladdr{http://people.bath.ac.uk/ac886/}

\author{Travis Schedler} 
\address{Department of Mathematics, Imperial College London, South Kensington Campus, London, SW7 2AZ, UK}
\email{t.schedler@imperial.ac.uk}
\urladdr{https://www.imperial.ac.uk/people/t.schedler}

\begin{abstract}
 We study GIT quotients $X_\theta=V\git_\theta G$ whose linearisation map defines an isomorphism between the group of characters of $G$ and the Picard group of $X_\theta$ modulo torsion. Our main result establishes that the Cox ring of $X_\theta$ is isomorphic to the semi-invariant ring of the $\theta$-stable locus in $V$. This applies to quiver flag varieties, Nakajima quiver varieties, hypertoric varieties, and crepant resolutions of threefold Gorenstein quotient singularities with fibre dimension at most one. As an application, we present a simple, explicit calculation of the Cox ring of the Hilbert scheme of $n$-points in the affine plane.
\end{abstract}

\maketitle
\section{Introduction}
The Cox ring of a projective variety $X$ over an algebraically closed field $\kk$ of characteristic zero with finite Picard rank was introduced in the work of Hu--Keel~\cite{HuKeel00}. The multigraded $\kk$-algebra $\Cox(X)$ is built from the global sections of all line bundles on $X$. When $\Cox(X)$ is finitely generated as a $\kk$-algebra, it was shown to encode completely the birational geometry of $X$, in the sense that every small birational model of $X$ can be reconstructed by variation of GIT quotient for the action of an algebraic torus on the affine variety whose coordinate ring is $\Cox(X)$. The Cox ring played a fundamental role in the work of Birkar--Cascini--Hacon--McKernan~\cite{BCHM}, and a comprehensive study of the properties of $\Cox(X)$ was carried out by Arzhantsev--Derenthal--Hausen--Laface~\cite{ADHL}.

In general, it can be difficult to describe the Cox ring explicitly, or to prove that it is a finitely generated $\kk$-algebra.  Here, we consider this question in the setting of a GIT quotient $X_\theta:=V \git_\theta G$, where $V$ is an affine variety, $G$ is a reductive group and $\theta$ is a character of $G$. In fact, we consider the relative version $\Cox(X_\theta/Y)$ for the natural morphism $X_\theta \to Y:=V \git G$ to the affine GIT quotient (see Ohta~\cite{Ohta22}). Our main theorem describes this relative Cox ring of $X_\theta$ over $Y$ as a \emph{semi-invariant ring}, and we apply this to describe the Cox ring of every class of GIT quotient that we studied in our earlier paper \cite{BCS23}.

To state the main result, let $\theta\in G^\vee$ be a generic character of $G$ and let $C$ be the GIT chamber in the vector space $G^\vee_{\QQ}:= G^\vee\otimes_\ZZ \QQ$ containing $\theta$.
%and let $\theta\in C$. 
Then there is a geometric quotient $V^\theta\to X_\theta$, where $V^\theta$ is the locus of $\theta$-stable points in $V$. The \emph{linearisation map} for $C$ is the $\QQ$-linear map 
\[
L_C\colon G^\vee_\QQ \longrightarrow \Pic(X_\theta/Y)_\QQ
\]
that typically assigns to each character $\chi$, the line bundle $L_\chi$ on $X^\theta$ whose pullback to $V^\theta$ is the trivial line bundle with $G$-linearisation $\chi$, often denoted $\chi \otimes \mathcal{O}_{V^\theta}$. An integral version
\[
\ell_C\colon G^\vee \longrightarrow \Pic(X_\theta/Y)/\text{torsion}
\]
of the linearisation map exists under mild assumptions on the action of $G$. We are primarily interested in the case where $\ell_C$ is an isomorphism, in which case $G^\vee\cong \ZZ^\rho$ for some $\rho\geq 1$. Then our main theorem can be stated as follows (see Theorem~\ref{thm:Coxringisom}):

\begin{theorem}
\label{thm:Coxringisomintro}
 Let $\theta\in G^\vee$ be generic, and let $C$ denote the GIT chamber containing $\theta$. Assume that $\ell_C$ is an isomorphism. Then $\Cox(X_\theta/Y)$ is isomorphic as a $\ZZ^\rho$-graded ring to $\Gamma(\mathcal{O}_{V^{\theta}})^{[G,G]}$.
\end{theorem}

For $H:=[G,G]$, the abelianisation $\Gab:= G/H$-acts on the $H$-invariant subring  of $\Gamma(\mathcal{O}_{V^{\theta}})$, which is therefore graded by the character group $\Gab^\vee$. When viewed as a multigraded ring in this way, we call $\Gamma(\mathcal{O}_{V^{\theta}})^H$ the \emph{semi-invariant ring} for the action of $G$ on the $\theta$-stable locus $V^\theta$.

We explain in Section \ref{sec:examples} that Theorem~\ref{thm:Coxringisomintro} applies to several classes of examples, including quiver flag varieties, Nakajima quiver varieties and hypertoric varieties (under mild conditions), and crepant resolutions of Gorenstein threefold quotient singularities with fibre  dimension at most one. In addition, Theorem~\ref{thm:Coxringisomintro} is used in recent work of Heuberger--Kalashnikov~\cite{HeubergerKalashnikov24} to compute a SAGBI basis for the Cox ring of some special moduli spaces of quiver representations, and in forthcoming work of Hubbard~\cite{Hubbard24} to compute generators and relations for the Cox ring of  hyperpolygon spaces in arbitrary dimension. 

 One might hope to phrase Theorem~\ref{thm:Coxringisomintro} in terms of the $[G,G]$-invariant subring of the coordinate ring $\Gamma(\mathcal{O}_V)$ rather than $\Gamma(\mathcal{O}_{V^{\theta}})$. However, the following 
 example shows that this is not possible. %the case.
 
\begin{example}
\label{exa:intro}
 The multiplicative group $\GGm$ acts on $\Spec \kk[a,b,c,d]$ with weights $(1,1,-1,-1)$, and it restricts to an action on the affine variety $V : (ad-bc=0)\subset \mathbb{A}^4$. There is an isomorphism $T^*\mathbb{P}^1\cong V\git_\theta G$ for nonzero $\theta\in G^\vee$, and the morphism $V\git_\theta G\to V\git G$ coincides with the minimal resolution $T^*\mathbb{P}^1\to Y:(xz=y^2)\subset \mathbb{A}^3$ of the $A_1$-singularity. The $[G,G]$-invariant subalgebra of $\Gamma(V)$ is the coordinate ring $\kk[a,b,c,d]/(ad-bc)$ of $V$ because $[G,G]$ is trivial, and yet the Cox ring of the toric surface $T^*\mathbb{P}^1$ over $Y$ is a polynomial ring in three variables, so $\Cox(T^*\mathbb{P}^1)\not\cong \Gamma(V)^{[G,G]}$. The key point in this example is that the $\theta$-unstable locus in $V$ is of codimension one.
 \end{example}

The GIT construction from Example~\ref{exa:intro} coincides with the GIT construction of $T^*\mathbb{P}^1$ as a quiver variety, a hypertoric variety and, after taking the product with $\mathbb{A}^1$, it defines a crepant resolution of the Gorenstein threefold quotient singularity of type $\frac{1}{2}(1,1,0)$. Thus, it is essential to consider $V^\theta$ rather than $V$ in  all of these three cases, discussed in Examples~\ref{exa:Nakajimaquivervars}, \ref{exa:hypertoric} and \ref{exa:threefold} below.

In  Corollary~\ref{c:cox-subring} we establish a generalisation of Theorem~\ref{thm:Coxringisomintro}, which relates the Cox ring to the semi-invariant ring even when the map $\ell_C$ fails to be injective. It can be shown that injectivity often fails for Nakajima quiver varieties which have no simple $\Pi(Q)$-module of the given dimension vector, and for crepant resolutions of general Gorenstein threefold quotient singularities.

Under additional assumptions on $V\git H$ (see Corollary~\ref{cor:weakerass}), we deduce that the Cox ring of $X_\theta$ over $Y$ is finitely-generated. When this is the case, the construction of Hu--Keel~\cite{HuKeel00} (see also \cite{Ohta22}) enables one to reconstruct all small  birational models of $X_\theta$ from the auxiliary GIT problem given by the action of the diagonalisable group $\Gab$ on the affinisation of $V^\theta\git H$. This question was motivated by our recent work \cite{BCS23} where, in addition to assuming that $L_C$ is an isomorphism and $X_\theta$ is normal, we also imposed conditions \cite[Condition~3.3(2),(3)]{BCS23} which allowed us to reconstruct all small birational models of $X_\theta$ over $Y$ from the \emph{original} GIT problem of $G$ acting on $V$. That provided an alternative approach to deducing that the Cox ring is finitely generated without assumptions on $V^\theta\git H$, at least when $X_\theta$ is $\mathbb{Q}$-factorial (see \cite[Corollary~3.16]{BCS23}).

We conclude by applying Theorem~\ref{thm:Coxringisomintro} to calculate explicitly the Cox ring of the Hilbert scheme of $n$-points on the affine plane for each $n\geq 1$. In the case of globally generated line bundles, this recovers a result of Haiman computing their global sections, without requiring his deep vanishing theorem (see \cite[Remark~A.4]{GanGinzburg}). Here 
we draw on results of Ginzburg~\cite[Appendix]{GanGinzburg} on the almost commuting variety to compute the $H$-invariant subring of $\Gamma(\mathcal{O}_{V^{\theta}})$, and hence we compute the global sections of all line bundles on the Hilbert scheme by Theorem~\ref{thm:Coxringisomintro}.

\medskip

\noindent \textbf{Acknowledgements.} The first two authors were supported in part by Research Project Grant RPG-2021-149 from The Leverhulme Trust. The first author was also supported by EPSRC grants EP-W013053-1 and EP-R034826-1.

\section{The semi-invariant ring}
 \label{sec:semi-invariant}

Throughout, $\kk$ denotes an algebraically closed field of characteristic zero, and a \textit{variety} is a reduced, irreducible separated scheme of finite type over $\kk$. Let $G$ be a reductive algebraic group acting algebraically on an affine variety $V$ with coordinate ring $\kk[V]$. Let $U\subseteq V$ be a $G$-invariant, open subset. For $\chi\in G^\vee:=\Hom(G,\GGm)$, we write $\Gamma(U,\mathcal{O}_U)_\chi$ for the vector space of semi-invariant functions of weight $\chi$, i.e.\ the space of functions $f\in \Gamma(U,\mathcal{O}_U)$ satisfying $g\cdot f = \chi(g) f$ for all $g\in G$. The \emph{semi-invariant ring} for the action of $G$ on $U$ is the $G^\vee$-graded ring $\bigoplus_{\chi\in G^\vee} \Gamma(U,\mathcal{O}_U)_\chi$.
  
  The commutator subgroup $H:=[G,G]$ fits into a short exact sequence
  \begin{equation}
      \label{eqn:groupses}
  1\longrightarrow H\longrightarrow G\longrightarrow \Gab\longrightarrow 1,
  \end{equation}
   where $\Gab$ is the abelianisation of $G$. It follows that $\Gab$ acts on the $H$-invariant subring $\Gamma(U,\mathcal{O}_U)^H$, making $\Gamma(U,\mathcal{O}_U)^H$ into a $\Gab^\vee$-graded ring.
   Applying $\Hom(-,\GGm)$ to \eqref{eqn:groupses} defines an isomorphism $\iota\colon \Gab^\vee\to G^\vee$ of character groups.
  
\begin{lemma}
\label{lem:equality}
The semi-invariant ring $\bigoplus_{\chi\in G^\vee} \Gamma(U,\mathcal{O}_U)_\chi$ is isomorphic to the $\Gab^\vee$-graded ring $\Gamma(U,\mathcal{O}_U)^H$. 
\end{lemma}
\begin{proof}
 Each $\chi\in G^\vee$ satisfies $\chi(h)=1$ for all $h\in H$, so every $\chi$-semi-invariant function in $\Gamma(U,\mathcal{O}_U)$ is $H$-invariant. Conversely, the action of $G$ on $\Gamma(U,\mathcal{O}_U)^H$ factors through the quotient $\Gab$, which is a reductive group.
 Thus, the ring 
 $\Gamma(U,\mathcal{O}_U)^H$ decomposes as a sum over all subspaces $\Gamma(U,\mathcal{O}_U)^H_{\chi}$ for $\chi \in \Gab^\vee \cong G^\vee$. 
 \end{proof}

Let $\theta\in G^\vee$. A point $v\in V$ is \emph{$\theta$-semistable} if there exists $m>0$ and a semi-invariant function $f\in \kk[V]$ of weight $m\theta$ such that $f(v)\neq 0$. The \emph{$\theta$-semistable locus} $V^\theta\subseteq V$ is the $G$-invariant, open subset of $\theta$-semistable points. The \emph{$\theta$-unstable locus} is the closed subset of $V$ obtained as the complement of the $\theta$-semistable locus.

\begin{proposition}
\label{prop:sect2}
    The morphism $V^{\theta}\git H \to V \git H$ is an open immersion, and the algebra
$\Gamma(\mathcal{O}_{V^\theta\git H})$ is isomorphic to    \begin{equation}
    \label{e:cox-fla}
\Gamma(\mathcal{O}_{V^\theta})^H %(V\git_\theta G)
    = \bigcap_{f \in \bigcup_{m \geq 1} \kk[V]_{m \theta} \setminus \{0\}} \kk[V]^H[f^{-1}].
    \end{equation}
    Moreover, this algebra is finitely-generated over $\kk$ if the union of the irreducible components of the $\theta$-unstable locus of codimension one in $V \git H$ is the vanishing set of a global function.
  \end{proposition}

 \begin{remark}
   Clearly, if $V^\theta \git H$ is affine (equivalently, $V^\theta$ is affine)
   %% This is true because the quotient $V \to V\git H$ is affine, and $V^\theta$ is the preimage of its image because if $v \in V$ is $\theta$-semistable and $f(v) \neq 0$ for $f \in \kk[V]_{m\theta}$, then the same is true for any point in the closure of $H \cdot v$ as $f$ is $H$-invariant.
   then the algebra $\Gamma(\mathcal{O}_{V^\theta})^H$ is finitely generated.
 \end{remark}
 
\begin{proof} 
 The isomorphism $\Gamma(\mathcal{O}_{V^\theta\git H})\cong \Gamma(\mathcal{O}_{V^\theta})^H$ holds since $V^\theta\to V^\theta\git H$ is a categorical quotient. We have 
$\Gamma(\mathcal{O}_{V^\theta})\cong 
\bigcap_{f \in \bigcup_{m \geq 1} \kk[V]_{m \theta}\setminus \{0\}} \kk[V][f^{-1}]$. Each $f \in \bigcup_{m \geq 1} \kk[V]_{m \theta}\setminus \{0\}$ is $H$-invariant so $D(f):=\Spec \kk[V][f^{-1}]$ is a $H$-stable affine open subset of $V^{\theta}$ and so $V^{\theta} \git H$ has an open affine cover given by the subsets $D(f) \git H:=\Spec (\kk[V][f^{-1}])^H\cong \Spec (\kk[V]^H[f^{-1}])$. This shows that $V^{\theta}\git H \to V \git H$ is an open immersion and that equation \eqref{e:cox-fla} holds as claimed.

For the statement on finite generation, first let $R$ be the integral closure of $\kk[V]^H$ in its field of fractions, so that $\Spec R$ is the normalisation of $V \git H$.  Since $R$ is a finitely-generated module over $\kk[V]^H$ and $\kk[V]^H$ is Noetherian, every $\kk[V]^H$-submodule of $R$ is finitely generated.  Set $S := \Gamma(\mathcal{O}_{V^\theta \git H}) \cap R \subseteq R$.  Then $S$ is a finitely-generated module over $\kk[V]^H$ and in particular a finitely-generated algebra over it.  

Next let $f \in \kk[V]^H$ be a function whose vanishing locus is the union of the $\theta$-unstable divisors.  Let $g \in \Gamma(\mathcal{O}_{V^\theta \git H})$. Then for some $n\geq 1$, the function $g f^n$ has no poles along any divisors, hence $g f^n \in S$.  We conclude that $\Gamma(\mathcal{O}_{V^\theta \git H})$ is contained in the ring $S[f^{-1}]$, which is finitely generated  over $\kk[V]^H$.  Conversely, we have $S \subseteq \Gamma(\mathcal{O}_{V^\theta \git H})$ by definition, and $f^{-1}$ is regular on $V^\theta \git H$, so
 $\Gamma(\mathcal{O}_{V^\theta \git H}) = S[f^{-1}]$ is finitely generated over $\kk[V]^H$ and hence is finitely generated over $\kk$. 
\end{proof}

\begin{corollary}
\label{cor:weakerass}
     If $V \git H$ is $S_2$ and there are no $\theta$-unstable divisors,
   then the graded ring \eqref{e:cox-fla} is isomorphic to $\kk[V]^H$.
\end{corollary}

\begin{proof}
 The $\theta$-unstable locus in $V\git H$ is of codimension at least two and, since $V\git H$ is $S_2$, every function defined on the $\theta$-semistable locus extends to a function defined on all of $V\git H$.
\end{proof}

\section{Cox rings of GIT quotients}
 Let $\theta\in G^\vee$. A point $v\in V^\theta$ is \emph{$\theta$-stable} if the stabiliser $G_v$ is finite and the orbit $G\cdot v$ is closed in $V^\theta$. A character $\theta\in G^\vee$ is \emph{effective} if $V^\theta$ is non-empty, and an effective character $\theta$ is \emph{generic} if every $\theta$-semistable point of $V$ is $\theta$-stable. 
For any effective $\theta\in G^\vee_{\QQ}$, the GIT quotient
\[
X_\theta:= V\git_\theta G:= \Proj \Big(\bigoplus_{j\geq 0} \kk[V]_{j\theta}\Big)
\]
 is the categorical quotient of the $\theta$-semistable locus $V^\theta$ by the action of $G$.  If $\theta\in G^\vee$ is generic, then the natural morphism $\pi\colon V^{\theta}\to V\git_\theta G$ is a geometric quotient. The variety $V\git_\theta G$ is projective over the affine quotient $Y:=V\git_0 G = \Spec \kk[V]^G$, and we write 
 \[
 \tau\colon X_\theta\longrightarrow Y
 \]
 for the structure morphism obtained by variation of GIT quotient. The relative Picard group is $\Pic(X_\theta/Y):= \Pic(X_\theta)/\tau^*\Pic(Y)$, and we write $\Pic(X_\theta/Y)_{\QQ}:= \Pic(X_\theta/Y)\otimes_\ZZ \QQ$.
 
 Assume now that $\theta\in G^\vee$ is generic. For $\chi\in G^\vee$, consider the $G$-equivariant line bundle $\chi\otimes \mathcal{O}_{V^\theta}$ on the $V^\theta$ given by equipping the trivial line bundle with the action of $G$ on each fibre given by $\chi$; explicitly, the $G$-action on $V^\theta$ lifts to the action on $V^\theta\times \mathbb{A}^1$ given by $g\cdot (v,z) = (g\cdot z, \chi^{-1}(g) z)$. It follows that the space of $G$-invariant sections of $\chi\otimes \mathcal{O}_{V^\theta}$  is isomorphic to the space $\Gamma(\mathcal{O}_{V^\theta})_\chi$ of $\chi$-semi-invariant functions on $V^\theta$. By descent \cite{Nevins08}, $\chi\otimes \mathcal{O}_{V^\theta}$ descends to a line bundle on $V\git_\theta G$ if and only if the stabiliser of each point $v \in V^{\theta}$ is in the kernel of $\chi$. Since all stabilisers are finite, and there are only finitely many conjugacy classes of such stabilisers by \cite[Corollaire~3]{Luna}, there is some multiple $m \chi\otimes \mathcal{O}_{V^\theta}$ for $m>0$ that descends. We define $L_{\chi}:= \frac{1}{m}L_{m \chi}\in \Pic(V\git_\theta G)_{\mathbb{Q}}$ to be the corresponding fractional line bundle. 
  
 \begin{lemma}
 \label{lem:sections}
 Let $\theta\in G^\vee$ be generic. Assume that the stabiliser at each point of $V^\theta$ is trivial.  There is an isomorphism of $G^\vee$-graded $\kk$-algebras
 \[
 \bigoplus_{\chi\in G^\vee} \Gamma(L_{\chi})\cong
  \bigoplus_{\chi\in G^\vee}
\Gamma(\mathcal{O}_{V^\theta})_\chi
 \]
 \end{lemma}
 \begin{proof}
  The stabilisers on $V^\theta$ are trivial, so $\chi\otimes\mathcal{O}_{V^\theta}$ descends to $L_\chi \cong \pi_*(\pi^*L_\chi)^G\cong \pi_*(\chi\otimes \mathcal{O}_{V^\theta})^G$. Adjunction for $\pi_*(-)^G$ and $\pi^*(-)$ (see, for example, \cite[Lemma~2.5]{Nevins08}) 
 implies that the $\kk$-vector space $\Gamma(L_\chi)\cong \Hom(\mathcal{O}_{X_\theta},L_\chi)$ is isomorphic to
$\Hom(\pi^*\mathcal{O}_{X_\theta},\chi\otimes \mathcal{O}_{V^\theta})^{G} \cong \Gamma(\mathcal{O}_{V^\theta})_\chi$
for each $\chi\in G^\vee$. These isomorphisms are obtained by restriction from the identification of rational functions on $X_\theta$ with $G$-invariant rational functions on $V^\theta$, so the ring structures agree.
  \end{proof} 
   
The set of effective fractional characters is a closed, convex cone in $G^\vee_{\QQ}:= G^\vee\otimes_{\ZZ} \QQ$ that admits a wall-and-chamber structure, called the \emph{GIT fan} (see \cite[Section~2.2]{BCS23}).  The set of generic stability parameters $\theta$ in $\Theta$ decomposes as the union of (GIT) chambers, each of which is the interior of a top-dimensional cone in the GIT fan. For $\theta\in G^\vee$ generic, let $C$ denote the GIT chamber with $\theta\in C$. The \emph{linearisation map} for $V\git_\theta G$ is the $\QQ$-linear map 
 \begin{equation}
 \label{eqn:linearisationmap}
 L_C\colon G^\vee_{\QQ}\longrightarrow \Pic(X_\theta/Y)_{\QQ}
 \end{equation}
 defined by setting $L_C(\chi):= L_\chi$. 
 
 Assume that the $G$-equivariant line bundle $\chi\otimes \mathcal{O}_{V^\theta}$ descends to an actual (rather than a fractional) line bundle $L_\chi$ for all $\chi\in G^\vee$. That is, we assume that the stabiliser at each closed point of $V^\theta$ is contained in the commutator subgroup $H$.  Then 
 we obtain a $\ZZ$-linear map
 \begin{equation}
     \label{eqn:ellC}
\ell_C\colon G^\vee\longrightarrow \Pic(X_\theta/Y)/\text{torsion}
 \end{equation}
 such that $\ell_C\otimes_{\ZZ} \QQ = L_C$. 
 %This happens if and only if the $G$-equivariant line bundle $\chi\otimes \mathcal{O}_{V^\theta}$ descends for all $\chi\in G^\vee$, that is, when the stabiliser at each closed point of $V^\theta$ is contained in the commutator subgroup $H$.
 %{Observe that the source of \eqref{eqn:ellC} is a free and finitely-generated $\ZZ$-module, since we assumed $G$ was connected and reductive.}
 %is trivial. 

 We are primarily interested in the case where the 
 %integral version of the linearisation map, 
 map $\ell_C$ from \eqref{eqn:ellC} is an isomorphism. In this case, both $G^\vee$ and $\Pic(X_\theta/Y)/\text{torsion}$ are finitely generated and free,
 %so $G$ is connected\footnote{\gwyn{What about stupid examples like $G = G^{\circ} \times S$ where $S$ is a finite simple group?}}, 
 and we can
 %of finitely generated and free $\ZZ$-modules. 
% When this is the case, 
choose a finite sequence of line bundles $(L_1, \dots,L_\rho)$ whose isomorphism classes provide 
%it 
a $\ZZ$-linear basis for $\Pic(X_\theta/Y)/\text{torsion}$. 
%for $\Pic(X_\theta/Y)/\text{torsion}$. 
Following Hu--Keel~\cite{HuKeel00} in the projective setting, and Ohta~\cite[Section~6.2]{Ohta22} in the relative setting, define the \emph{Cox ring} of $X_\theta$ over $Y$ to be the $\ZZ^\rho$-graded $\kk$-algebra
 \[
 \Cox(X_\theta/Y) :=\bigoplus_{(m_1, \dots, m_\rho)\in \mathbb{Z}^\rho} \Gamma(L_1^{m_1}\otimes \dots \otimes L_{\rho}^{m_\rho}\big).
 \]
 If $Y=\Spec \kk$, 
 %then $\tau_*\big(L_1^{m_1}\otimes \dots \otimes L_{\rho}^{m_\rho}\big)\cong \Gamma(L_1^{m_1}\otimes \dots \otimes L_{\rho}^{m_\rho})$ as $\kk$-vector spaces, and
 then we simply write $\Cox(X_\theta)$ for the Cox ring of $X_\theta$ over $Y$.
 
 The question as to whether $\Cox(X_\theta/Y)$ is finitely generated as an $\mathcal{O}_Y$-algebra is independent of the choice of $\ZZ$-basis for $\Pic(X_\theta/Y)/\text{torsion}$. Moreover,  $\Cox(X_\theta/Y)$ is finitely generated as an $\mathcal{O}_Y$-algebra if and only if it is finitely generated as a $\kk$-algebra because the coordinate ring $\kk[V]^G$ of $Y$ is itself a finitely generated $\kk$-algebra. 
 % and $\kk[V]$ the $\mathcal{O}_Y$-module $\tau_*(L_1^{m_1}\otimes \dots \otimes L_{\rho}^{m_\rho})$ is finitely generated for each $(m_1, \dots, m_\rho)\in \ZZ^\rho$ because $\tau$ is proper.}

  \begin{theorem}
\label{thm:Coxringisom}
 Let $\theta\in G^\vee$ be generic, and let $C$ denote the GIT chamber containing $\theta$. Assume that $\ell_C$ is an isomorphism. Then $\Cox(X_\theta/Y)$ is isomorphic as a $\ZZ^\rho$-graded ring to 
 $\Gamma(\mathcal{O}_{V^{\theta}})^{H}$.
  \end{theorem}

\begin{proof}
%We construct a commutative diagram of graded ring homomorphisms 
% \begin{equation}
%     \label{eqn:Coxdiagram}
%\begin{tikzcd}
%\kk[V]^{H} \ar[dr,"\phi"]  \ar[d, swap, "j^*"] &  \\
%\Gamma(\mathcal{O}_{V^{\theta}})^{H} \ar[r,"\psi"] & \Cox(X_\theta/Y),
%\end{tikzcd}
%\end{equation}
%  where $j\colon V^\theta\hookrightarrow V$ is the open immersion and $\psi$ is an isomorphism. 
First, consider the isomorphism of graded $\kk$-algebras 
 \begin{equation}
     \label{eqn:gradedone}
 \psi_0\colon \Gamma(\mathcal{O}_{V^{\theta}})^{H} 
 \longrightarrow 
  \bigoplus_{\chi\in G^\vee} \Gamma(L_{\chi})
  \end{equation}
  obtained as the composition of the graded $\kk$-algebra isomorphisms from Lemma~\ref{lem:equality} and Lemma~\ref{lem:sections}, where the isomorphism of character groups $\iota\colon \Gab^\vee\to G^\vee$ identifies the natural $\Gab^\vee$-grading on the domain of $\psi_0$ with the $G^\vee$-grading on the codomain. The isomorphism $\ell_C$ identifies $G^\vee$ with $\Pic(X_\theta/Y)/\text{torsion}\cong \ZZ^\rho$, so we obtain an isomorphism  
\begin{equation}
    \label{eqn:gradedtwo}
\bigoplus_{\chi\in G^\vee} \Gamma(L_{\chi})\longrightarrow \Cox(X_\theta/Y)
\end{equation}
 of $\ZZ^\rho$-graded rings. Define $\psi$ to be the composition of the graded $\kk$-algebra homomorphisms \eqref{eqn:gradedone} and \eqref{eqn:gradedtwo}. % and define $\phi:=\psi\circ j^*$ to ensure that the diagram commutes. 
 %Since \eqref{eqn:gradedone} is an isomorphism, $\psi$ is an isomorphism (resp.\ is injective or surjective) if and only if \eqref{eqn:gradedtwo} is an isomorphism (resp.\ is injective or surjective) which holds if and only if $\ell$ is bijective (resp.\ is injective or surjective).  
 \end{proof}

Putting this together with the results of the previous subsection, we conclude:

\begin{corollary} 
\label{cor:Cox}
Suppose that the assumptions of Theorem~\ref{thm:Coxringisom} hold. Then $\Cox(X_\theta/Y)$ is given by \eqref{e:cox-fla}. Moreover:
    \begin{enumerate}
        \item If the union of the irreducible components of the $\theta$-unstable locus of codimension one in $V\git H$ is the vanishing locus of a global function, then $\Cox(X_\theta/Y)$ is a finitely-generated $\kk$-algebra.
        \item If the quotient $V \git H$ is $S_2$ and the $\theta$-unstable locus has codimension at least two, then $\Cox(X_\theta/Y) \cong \kk[V]^H$.
    \end{enumerate}
\end{corollary}
  
The hypotheses in Theorem~\ref{thm:Coxringisom} are stronger than necessary; they were chosen for simplicity, and they are satisfied in the examples from Section~\ref{sec:examples}. We now 
%In this section, 
explain the general situation.
%Let $G$ be a (possibly disconnected) reductive group, still acting on an affine variety $V$.  Assume that $\theta \in C \subseteq G^\vee_\QQ$ is generic.

We continue to assume that $\theta\in G^\vee$ is generic, with $C$ the chamber containing $\theta$. This gives rise to the linear map $L_C$ as described above. The restriction of $L_C$ to $G^\vee$ factors through $G^\vee/\text{torsion}$, which is finitely generated and free. Let $\Lambda \leq G^\vee$ be the full-rank sublattice of characters $\chi$ such that $\chi \otimes \mathcal{O}_{V^\theta}$ descends to $X_\theta$, i.e., the stabilisers of $G$ on $V^\theta$ are in the kernel of $\chi$. Then $\Lambda$ is a finite-index subgroup of $G^\vee/\text{torsion}$. The map $L_C$ yields a homomorphism:
\[
\ell_C: \Lambda \twoheadrightarrow  \Pi :=  L_C(\Lambda) \subseteq
\Pic(X_\theta/Y)/\text{torsion}.
\]
Observe that $\Pi$ is the subgroup of $\Pic(X_\theta/Y)/\text{torsion}$ of classes obtained by descent. Let $\Lambda'$ be a complement in $\Lambda$ to $\ker(\ell_C)$. Then $\ell_C$ restricts to an isomorphism $\Lambda' \to \Pi$.  Define
\begin{equation}
H' := \bigcap_{\chi \in \Lambda'} \ker(\chi) \leq G.
\end{equation}
Note that $H \leq H' \leq G$. Since $H=[G,G]$ is reductive, and $G/H=\Gab$ is reductive and abelian, we deduce that $H'/H$ and hence $H'$ is reductive. 
%and $H'$ is reductive\footnote{\ali{How to see this last assertion?}}.
Then the proof of Theorem \ref{thm:Coxringisom} yields an isomorphism
\begin{equation} \label{e:h0-iso-general}
\Gamma(\mathcal{O}_{V^\theta})^{H'} \cong \bigoplus_{\chi \in \Lambda'} \Gamma(\mathcal{O}_{V^\theta})_\chi \cong \bigoplus_{L \in \Pi} \Gamma(L) \subseteq \Cox(X_\theta/Y).
\end{equation}
 We conclude, without assumptions on $L_C$:

\begin{corollary}\label{c:cox-subring}
The subring of $\Cox(X_\theta/Y)$ 
graded by $\Pi$
%classes of $\Pic(X_\theta/Y)/\text{torsion}$ obtained by descent 
is isomorphic to $\Gamma( \mathcal{O}_{V^\theta})^{H'}$. If every class in $\Pic(X_\theta/Y)/\text{torsion}$ is obtained by descent from a character of $G$, then $\Cox(X_\theta/Y) \cong \Gamma(\mathcal{O}_{V^\theta})^{H'}$.   
%, for $H'$  an explicit\footnote{\ali{I'm suggesting to define $H'$ in an equation as above so we can refer to it and cut half of this sentence.}} reductive subgroup of $G$ containing the commutator subgroup.
\end{corollary}
%\travis{The left-hand side of \eqref{e:h0-iso-general} is finitely generated\footnote{\ali{This requires the condition in the second sentence of Prop~2.2, as stated in the corollary to follow.}} by replacing $H$ by $H'$ in Proposition~\ref{prop:sect2}. 
This 
%The inclusion
%Since the last inclusion 
%in \eqref{e:h0-iso-general} 
%is of the subalgebra concentrated in a subgroup of the grading group, 
allows us to deduce a condition for the Cox ring to be finitely generated:
%Suppose now that the LHS of \eqref{e:h0-iso-general} is finitely-generated. To deduce from this that the Cox ring is finitely generated, we need 
%$\Pi \leq \Pic(X_\theta/Y)/\text{torsion}$ to have finite index. This happens if and only if $L_C$ is surjective. Combining this with the main results in this section gives the following:

%i.e., every line bundle on $X_\theta$ is obtained by descent from $V_\theta$, up to tensoring by a bundle representing a torsion class in $\Pic(X_\theta/Y)$.  
%We deduce the following from the proofs of our main results:
\begin{corollary} Suppose that $L_C$ is surjective.
%\begin{enumerate}
    %\item  \item 
    If %$V^\theta \git H'$ is $S_2$ and 
    the union of the components of the $\theta$-unstable locus of $V\git H'$ which have codimension one  is the vanishing locus of a global function,
    %Cartier divisor (e.g., $V$ is normal and $V^\theta$ has complement of codimension at least two), 
    then $\Cox(X_\theta/Y)$ is finitely generated.
%for $H'$ as above.
 %   \end{enumerate}
\end{corollary}
\begin{proof}
 Since $V$ is assumed to be an affine variety, it is reduced and irreducible. Hence so too is $X_{\theta}$. This implies that  $\Cox(X_\theta/Y)$ has no zero divisors since, if it did, there would be a homogeneous zero divisor, and the global sections of every line bundle embed inside the ring of rational functions on $X_\theta$, compatibly with the multiplication. 
    %$\Gamma( \mathcal{O}_{X_\theta})$-module. 
    % The Cox ring of X has no zero divisors, because the fraction field of X has no zero-divisors. It now follows from the general statement in [ADHL, Coro 1.1.2.6] that Cox(X) is finitely generated iff a subalgebra graded by a subgroup of finite index in the grading group is finitely generated. 
    The last inclusion in \eqref{e:h0-iso-general} is the subalgebra concentrated in degrees $\Pi \leq \Pic(X_\theta/Y)/\text{torsion}$,  which is a subgroup of finite index if and only if $L_C$ is surjective. As a result, \cite[Corollary 1.1.2.6]{ADHL} ensures that, when $L_C$ is surjective, $\Cox(X_\theta/Y)$ is a finitely-generated algebra over $\kk$ if and only if $\Gamma(\mathcal{O}_{V^\theta})^{H'}$ is. The result then follows from Proposition~\ref{prop:sect2}.
    %over the subalgebra graded by $\Pi$. This inclusion is an equality if and only if $\Pi=\Pic(X_\theta/Y)/\text{torsion}$. 
\end{proof}

In particular, when all the assumptions of the corollary hold, we can use the action of $G/H'$ on the affinisation of $V^\theta \git H'$ to obtain all small birational models of $X_\theta$.

\begin{remark}
  We have assumed throughout that $V$ is reduced and irreducible. This can be relaxed as follows (cf.~Example \ref{exa:threefold} and Section \ref{s:hilbert} below).  Take  a GIT quotient $X_\theta = V\git_\theta G$ with $\theta$ generic and let $X_\theta^0$ be the reduced subscheme of an irreducible component  of this quotient. If $G$ is connected, then 
  %Let us suppose further that 
  the preimage of $X_\theta^0$ in $V$ is irreducible.
  %(this is guaranteed i). 
  Let $V^0 \subseteq V$ be the reduced subscheme of this irreducible component. Then $V^0\git_\theta G \cong X_\theta^0$.
  %Alternatively, if the preimage of $X_\theta^0$ is disconnected, we could take $V^0$ to be the reduced subscheme of an irreducible component and replace $G$ by the stabiliser $G^0$ of $V^0$. If $X_\theta^0$ is normal, then it follows from Zariski's main theorem that the map    $V^0\git_\theta G^0 \to X_\theta^0$ is an isomorphism (since it is bijective).     
  In particular we can use the theory developed above to study the reduced subscheme of every irreducible component of the quotient $V\git_\theta G$.
  %(at least if $G$ is connected or the component is normal).}
\end{remark}

 \section{Classes of examples}
 \label{sec:examples}
 We now apply Theorem~\ref{thm:Coxringisom} to describe the Cox ring as a semi-invariant ring for several large classes of examples. A common theme here is that all of these examples are (relative) Mori Dream Spaces, or equivalently, their Cox rings are finitely generated $\kk$-algebras. 

 \begin{comment}
 \ali{First, it is convenient to recall the construction of moduli spaces of quiver representations. Let $Q$ be a finite quiver with vertex set $Q_0$ and arrow set $Q_1$. Write $\head,\tail\colon Q_1\to Q_0$ for the maps that assign to each arrow $a$ in $Q$, the vertex at the head and tail of $a$.  For any dimension vector $\alpha=(\alpha_i)\in \NN^{Q_0}$, the group $\GL_\alpha:=\prod_{i\in Q_0} \GL(\alpha_i)$ acts naturally on the vector space 
 \[
 \Rep(Q,\alpha)=\bigoplus_{a\in Q_1} \Hom\big(\kk^{\alpha_{\tail(a)}}, \kk^{\alpha_{\head(a)}}\big)
 \]
 of representations of $Q$ of dimension $\alpha$. In this context, the commutator subgroup $H$ of $\GL_\alpha$ is traditionally denoted $\SL_\alpha:=\prod_{i\in Q_0} \SL(\alpha_i)$. The diagonal scalar subgroup $\kk^\times\subset \GL_\alpha$ acts trivially, leaving a faithful action of $G:=\GL_\alpha/\kk^\times$. For $\theta=(\theta_i)\in G^\vee$, define the 
 %for the character satisfying $\theta(g) = \prod_{i\in Q_0} \det(g_i)^{\theta_i}$ for $g=(g_i)\in G$. Following King~\cite{King94}, the 
 GIT quotient 
  \[ \mathcal{M}_\theta(Q,\alpha):=\Rep(Q,\alpha)\git_\theta G.
 \]
 If $\alpha$ is indivisible and $\theta$ is generic, then  King~\cite{King94} showed that $\mathcal{M}_\theta(Q,\alpha)$ is the fine moduli space of $\theta$-stable representations of $Q$ that have dimension $\alpha$.}
 \end{comment}
 
\begin{example}[\textbf{Quiver flag varieties}]
  Let $Q$ be a finite, acyclic quiver with a unique source vertex denoted $0\in Q_0$, and let $\alpha=(\alpha_i)\in \NN^{Q_0}$ be a dimension vector satisfying $\alpha_0=1$. The group $\GL_\alpha:=\prod_{i\in Q_0} \GL(\alpha_i)$ acts by change of basis on the vector space $\Rep(Q,\alpha)$ of representations of $Q$ of dimension $\alpha$. In this context, the commutator subgroup $H$ of $\GL_\alpha$ is traditionally denoted $\SL_\alpha:=\prod_{i\in Q_0} \SL(\alpha_i)$. If we write $G:=\GL_\alpha/\GGm$ for the quotient by the diagonal scalar subgroup, then the unique character $\vartheta=(\vartheta_i)\in G^\vee$ satisfying $\vartheta_i=1$ for $i\neq 0$ is generic. The fine moduli space $\mathcal{M}_\vartheta(Q,\alpha):=\Rep(Q,\alpha)\git_\vartheta G$ of quiver representations is a smooth, projective variety called a \emph{quiver flag variety}~\cite{Craw11}. Let $C$ denote the GIT chamber containing $\vartheta$. The integral version $\ell_C$ of the linearisation map is an isomorphism and the $\vartheta$-unstable locus has codimension at least two by \cite[(3.1), Lemma 3.7 and the proof of Proposition~3.1]{Craw11}. It follows from Corollary~\ref{cor:Cox} that 
 \begin{equation}
     \label{eqn:Coxquiverflag}
\text{Cox}\big(\mathcal{M}_\vartheta(Q,\alpha)\big) \cong \kk[\Rep(Q,\alpha)]^{\SL_\alpha}
 \end{equation}
  as a $\ZZ^{\vert Q_0\vert -1}$-graded ring. Since $\SL_\alpha$ is reductive, the Cox ring is a finitely generated $\kk$-algebra, so $\mathcal{M}_\theta(Q,\alpha)$ is a Mori Dream Space by Hu--Keel~\cite{HuKeel00} (this was first proved in \cite[Proposition~3.1]{Craw11}). It follows that the isomorphism \eqref{eqn:Coxquiverflag} holds for any generic $\theta\in G^\vee$ such that the line bundle $\ell_C(\theta)$ lies in the movable cone of $\mathcal{M}_\vartheta(Q,\alpha)$.
\end{example}

  \begin{remarks}
  \begin{enumerate}
      \item The graded ring \eqref{eqn:Coxquiverflag} is the \emph{semi-invariant ring} of the quiver $Q$ and choice of vector $\alpha$. This ring has been much studied in the literature; see, for example, Schofield~\cite{Schofield91}, Derksen--Weyman~\cite{DW00,DW17}, Domokos--Zubkov~\cite{DZ01} and Schofield--Van den Bergh~\cite{SVdB01}. 
   \item More generally, for any acyclic quiver $Q$, dimension vector $\alpha$ and generic character $\vartheta\in G^\vee$, the isomorphism \eqref{eqn:Coxquiverflag} describes the Cox ring of $\mathcal{M}_\vartheta(Q,\alpha)$ if the $\vartheta$-unstable locus is of codimension at least two in $\Rep(Q,\alpha)$. Indeed, 
   %that is, $\alpha$ is \emph{$\theta$-amply stable} in the sense of Reineke et al.~\cite{FRS21}. Indeed, this condition implies that 
   the integral version of the linearisation map is an isomorphism (see \cite[Proposition~3.1]{FRS21}) and the claim now follows from Corollary~\ref{cor:Cox}.
  \end{enumerate}
  \end{remarks}
 
  \begin{example}[\textbf{Nakajima quiver varieties}]
  \label{exa:Nakajimaquivervars}
  Consider a finite graph with vertices $0,1,\dots, r$, and dimension vectors $\mathbf{v}, \mathbf{w}\in \mathbb{N}^{r+1}$ with $\mathbf{w}\neq 0$. Add a framing vertex $\infty$, and $w_i$ edges joining $\infty$ to vertex $i$. Let $Q$ denote the doubled quiver of this new graph with vertex set $Q_0=\{\infty, 0, 1, \dots, r\}$, and define $\alpha:=(1,\mathbf{v})\in \mathbb{N}^{r+2}$. 
 % Let $Q$ denote the framed, doubled quiver of the graph, so $Q_0=\{\infty, 0, 1, \dots, r\}$ where $\infty$ is a framing vertex. 
 The reductive group $G:=\prod_{0\leq k\leq r} \GL(v_k)$ acts by change of basis on the space $\Rep(Q,\alpha)$ of representations of $Q$ of dimension vector $\alpha$. Write $\SL:=\prod_{0\leq k\leq r} \SL(v_k)$ for the commutator subgroup.  Since $Q$ is a doubled quiver, the action of $G$ is Hamiltonian for the natural symplectic form on $\Rep(Q,\alpha)$, giving rise to a moment map $\mu\colon \Rep(Q,\alpha)\to \mathfrak{g}^*$. Following Nakajima~\cite{Nak1994, NakDuke98}, after replacing any $\theta\in G^\vee_\QQ$ %$\otimes_{\ZZ} \QQ$ 
  by a positive multiple if necessary, the  \emph{Nakajima quiver variety} is defined to be
  \begin{equation}
      \label{eqn:NakajimaGIT}
\mathfrak{M}_\theta(\mathbf{v},\mathbf{w}):= \mu^{-1}(0)\git_\theta G,
   \end{equation}
  taken with the underlying reduced scheme structure (see \cite[Section~4.1]{BCS23} for notation and details). Bellamy--Schedler~\cite{BS21} show that $\mathfrak{M}_\theta(\mathbf{v},\mathbf{w})$ is a normal variety. Assume that a simple representation of the preprojective algebra $\Pi(Q)$ of dimension $\alpha$ exists; following Crawley-Boevey~\cite{CBmomap}, we write this condition as $\alpha\in \Sigma_0$. Then \cite[Theorem~1.2]{CBmomap} implies that $\mu^{-1}(0)$ is reduced and irreducible, i.e.\ it is a variety. As shown in \cite[Proposition~4.2]{BCS23}, work of McGerty--Nevins~\cite{QuiverKirwan} implies that for any generic $\theta\in G^\vee$ in a GIT chamber $C$, the integral version $\ell_C$ of the linearisation map is an isomorphism. Then
 \[
\text{Cox}\big(\mathfrak{M}_\theta(\mathbf{v},\mathbf{w})/\mathfrak{M}_0(\mathbf{v},\mathbf{w})\big) \cong 
\Gamma\big(\mathcal{O}_{\mu^{-1}(0)^\theta}\big)^{\SL}
 \]
  as a $\ZZ^{\vert Q_0\vert - 1}$-graded ring by Theorem~\ref{thm:Coxringisom}. Namikawa~\cite{NamikawaPoissondeformations} (see also \cite[Corollary~4.8]{BCS23}) shows that $\mathfrak{M}_\theta(\mathbf{v},\mathbf{w})$ is a Mori Dream Space over $\mathfrak{M}_0(\mathbf{v},\mathbf{w})$, so this Cox ring is a finitely generated $\kk$-algebra. 
\end{example}

%\begin{remark}
%  \ali{Hubbard~\cite{Hubbard24} uses Example~\ref{exa:Nakajimaquivervars} to calculate explicitly the Cox ring of any hyperpolygon space which, by construction, is a Nakajima quiver variety for a star-shaped quiver.}  
%\end{remark}

  \begin{example}[\textbf{Hypertoric varieties}]
  \label{exa:hypertoric}
 Here we work over the field $\mathbb{C}$ of complex numbers. For $n, r\in \mathbb{N}$, consider a short exact sequence of finitely generated and free abelian groups
 \[
 0\longrightarrow \ZZ^{n-r}\stackrel{B}{\longrightarrow} \ZZ^n\stackrel{A}{\longrightarrow}\ZZ^r
 \longrightarrow 0,
 \]
 where no row of the matrix $B$ is zero, and where $A$ is a unimodular matrix. Consider the action of $G:= (\mathbb{C}^\times)^r$ on the complex symplectic vector space $T^*\mathbb{C}^n = \mathbb{C}^n\times (\mathbb{C}^n)^*$, where the matrix $(A,-A)$ records the weights of the action. This action is Hamiltonian for the natural symplectic form on $T^*\mathbb{C}^n$, giving a moment map $\mu\colon T^*\mathbb{C}^n\to \mathfrak{g}^*$. For any charcter $\theta\in G^\vee$, the corresponding \emph{hypertoric (or toric hyperk\"ahler) variety} is
 \[
 X_\theta:= \mu^{-1}(0)\git_\theta G. 
 \]
 For generic $\theta$ in a GIT chamber $C$, the variety $X_\theta$ is non-singular \cite[Proposition~6.2]{HS02}, while the proof of \cite[Theorem~5.1]{BCS23} shows that the integral version $\ell_C$ of the linearisation map is an isomorphism of finitely generated and free abelian groups. Since $G\cong \Gab$, Theorem~\ref{thm:Coxringisom} shows that
 \[
 \Cox(X_\theta/Y)\cong \Gamma\big(\mathcal{O}_{\mu^{-1}(0)^\theta}\big)%CC[\mu^{-1}(0)^\theta\big]
 \]
 as $G^\vee$-graded algebras for $Y=X_0$. Again, note that $\Cox(X_\theta/Y)$ is a finitely generated $\CC$-algebra because $X_\theta$ is a Mori Dream Space over $Y$ \cite[Corollary~3.16, Theorem~5.1]{BCS23}.
 \end{example}
 
  \begin{example}[\textbf{Moduli spaces of $\Gamma$-constellations}]
  \label{exa:threefold}
  Let $\Gamma\subset \SL(3,\kk)$ be a finite subgroup, let $\text{Irr}(\Gamma)$ denote the set of isomorphism classes of irreducible representations of $\Gamma$, and write $\rho_0$ for the trivial representation of $\Gamma$. A \emph{$\Gamma$-constellation} is a $\Gamma$-equivariant coherent sheaf $F$ on $\mathbb{A}^3$ whose space of sections $\Gamma(F)$ is isomorphic as a $\kk[\Gamma]$-module to the regular representation $R$ of $\Gamma$. If $W$ is the three-dimensional representation of $\Gamma$ given by the inclusion $\Gamma\subset \SL(3,\kk)$, then the affine scheme $Z :=
  \big\{B\in \Hom_{\kk[\Gamma]}(R,W\otimes R) \mid B\wedge B = 0\big\}$ parametrises $\Gamma$-constellations, while isomorphism classes of $\Gamma$-constellations are precisely the orbits under the natural action of the group $\text{Aut}_\Gamma(R)\cong \prod_{\rho\in \text{Irr}(\Gamma)} \GL(\dim \rho)$ \cite[Section~5.3]{Craw01}. The quotient  $G:=\text{Aut}_\Gamma(R)/\GGm\cong \prod_{\rho\neq \rho_0} \GL(\dim\rho)$ by the diagonal scalar subgroup acts faithfully, and we write $\SL:= \prod_{\rho\neq \rho_0} \SL(\dim\rho)$ for the commutator subgroup. 
  
  For any generic $\theta\in G^\vee$, the GIT quotient $\mathcal{M}_\theta:= Z \git_\theta G$ is the fine moduli space of $\theta$-stable $\Gamma$-constellations. Moreover, by \cite[Theorem~2.5]{CrawIshii04}, $\mathcal{M}_\theta$ is smooth and irreducible, so there exists a unique irreducible component $Z_0$ of $Z$ such that $\mathcal{M}_\theta = V \git_\theta G$ with $V = (Z_0)_{\mathrm{red}}$. A result of Craw--Ishii~\cite[Section~2]{CrawIshii04} shows that the morphism
   \[
   \tau_\theta\colon \mathcal{M}_\theta\longrightarrow \mathbb{A}^3/\Gamma
   \]
    sending a $\Gamma$-constellation to its supporting $\Gamma$-orbit is a projective, crepant resolution of singularities. If every nontrivial conjugacy class of $\Gamma$ is junior in the sense of Ito--Reid~\cite{ItoReid96}, then the proof of \cite[Lemma~6.2]{BCS23} shows that the map $\ell_C$ 
    %of finitely generated and free abelian groups 
    associated to the GIT chamber $C$ containing $\theta$ is an isomorphism.  The group $\Pic(\mathbb{A}^3/\Gamma)$ is trivial,
    %by Benson~\cite{Benson93},  
    so Theorem~\ref{thm:Coxringisom} implies that
 \[
\text{Cox}(\mathcal{M}_\theta) \cong \Gamma(\mathcal{O}_{V^\theta})^{\SL}
 \]
  as a $\Gab^\vee$-graded ring. 
  %In this context, 
  Grab~\cite{Grab19} showed that $\Cox(\mathcal{M}_\theta)$ is a finitely generated $\kk$-algebra, so $\mathcal{M}_\theta$ is a Mori Dream Space over $\mathbb{A}^3/\Gamma$; an independent proof appears in \cite[Corollary~3.16 and Theorem~6.9]{BCS23}. 
  \end{example}

 \section{The Cox ring of the Hilbert scheme of points in the plane}\label{s:hilbert}
 We conclude by computing explicitly the Cox ring of the Hilbert scheme of $n$-points in the affine plane for $n\geq 1$. 
 
  Consider the graph with one vertex and one loop, and choose dimension vectors $\mathbf{v}=n$ and $\mathbf{w}=1$, as in Example~\ref{exa:Nakajimaquivervars}. The commutator of $G = \GL(n)$ is $H=\SL(n)$, and  we identify the character group of the quotient torus $\Gab\cong \GGm$ with $\ZZ$ in such a way that the isomorphism $\iota\colon \ZZ\to G^\vee$ from Section~\ref{sec:semi-invariant} sends $\theta$ to the character $\chi_\theta$ satisfying $\chi_\theta(g)=\det(g)^\theta$ for $g\in \GL(n)$. With this sign convention, it is well known \cite{Craw23,  NakajimaBook} that the Nakajima quiver variety $\mathfrak{M}_\theta(\mathbf{v},\mathbf{w})=\mu^{-1}(0)\git_\theta G$ for $\theta>0$ coincides with $\text{Hilb}^{[n]}(\mathbb{A}^2)$, both as a variety and as a fine moduli space, and moreover, the $G$-invariant subalgebra $\kk[\mu^{-1}(0)]^G$ is isomorphic to the coordinate ring of the $n^{th}$ symmetric product $\text{Sym}^n(\mathbb{A}^2)$.  (Note that Ginzburg~\cite{GinzburgQuiverVarieties} uses the opposite sign convention in defining $\iota$.)
 
 Gan--Ginzburg~\cite[Theorem~3.3.3]{GanGinzburg} show that the affine scheme $\mu^{-1}(0)$ %for the GIT construction from \eqref{eqn:NakajimaGIT}
 is reduced, equidimensional and has irreducible components $\mathscr{M}_0, \dots, \mathscr{M}_r$, while \cite[(2.8.1)]{GanGinzburg} constructs a closed immersion $\boldsymbol{\varepsilon} \colon \mathbb{A}^{2n} \hookrightarrow \mu^{-1}(0)$ that factors via $\mathscr{M}_0$. In addition, in the appendix to \emph{loc.~cit.}, Ginzburg studies polynomial functions on $\mu^{-1}(0)$. To state the result, consider the action of the symmetric group $S_n$ on $\kk[\mathbb{A}^{2n}]$
 %= \kk[x_1, \dots, x_n, y_1, \dots, y_n]$ 
 that permutes simultaneously the first $n$ variables and the last $n$ variables.
 %We define $A^0 = P^{S_n}$, $A^1 = P^{sign}$ and $A^k = A^1 \cdots A^1$ ($k$-times).establishes the following:}

\begin{lemma}[Ginzburg]
\label{lem:Ginzburg}  
For $\theta>0$, pullback along $\boldsymbol{\varepsilon}$ defines an isomorphism from 
%the space 
$\kk[\mu^{-1}(0)]_\theta$ 
%is isomorphic, via $\boldsymbol{\varepsilon}^*$, 
to the subspace in $\kk[\mathbb{A}^{2n}]$ spanned by products $\prod_{1\leq i\leq \theta} p_i$, where each $p_i$ is sign-semi-invariant for the action of $S_n$. Moreover:
\begin{enumerate}
\item[\one] the inclusion $\mathscr{M}_0\hookrightarrow \mu^{-1}(0)$ induces an isomorphism $\kk[\mu^{-1}(0)]_\theta \cong \kk[\mathscr{M}_0]_\theta$;
\item[\two] the $\theta$-stable locus of $\mu^{-1}(0)$ is contained in $\mathscr{M}_0$, so $ \Hilb^{[n]}(\mathbb{A}^2)\cong \mathscr{M}_0\git_\theta \GL(n)$; and
 \item[\three] if $f\in \kk[\mu^{-1}(0)]_\theta$ is non-zero, then $\kk[\mu^{-1}(0)][f^{-1}] = \kk[\mathscr{M}_0][f^{-1}]$. 
 \end{enumerate}
\end{lemma}

 \begin{remark}
Since the scheme $\mu^{-1}(0)$ is equidimensional, the $\theta$-unstable locus in $\mu^{-1}(0)$ is of codimension zero. 
\end{remark}     
     
Let $\GL(n)$ act on $\kk[\GGm] = \kk[t^{\pm 1}]$ as $g \cdot t = \det(g) t$. Pullback along $\boldsymbol{\varepsilon}$ induces an embedding of graded rings $\kk[\mathscr{M}_0]^{\SL} \cong \kk[\mathscr{M}_0 \times \GGm]^{\GL} \hookrightarrow \kk[\mathbb{A}^{2n} \times \GGm]$. The action of $S_n$ on $\kk[\mathbb{A}^{2n}]$ extends to an action on $\kk[\mathbb{A}^{2n} \times \GGm]:=\kk[\mathbb{A}^{2n}]\otimes_\kk \kk[t^{\pm 1}]$ by setting $\sigma(t) := \sign(\sigma) t$ for $\sigma \in S_n$. Ginzburg shows that for $\theta < 0$, any $\theta$-semi-invariant function on $\mu^{-1}(0)$ vanishes when restricted to $\mathscr{M}_0$, so Lemma~\ref{lem:Ginzburg} implies that the semi-invariant ring $\kk[\mathscr{M}_0]^{\SL}$ is isomorphic to the subring
$\kk[\mathbb{A}^{2n}]^{S_n} \oplus \bigoplus_{m>0} \big(\kk[\mathbb{A}^{2n}]_{\sign} t \big)^m$ of  $\kk[\mathbb{A}^{2n} \times \GGm ]$. It is convenient to  set $A^0:= \kk[\mathbb{A}^{2n}]^{S_n}$, $A^1 = \kk[\mathbb{A}^{2n}]_{\sign}$ and $A^m:= A^1 \cdots A^1$ ($m$-times) for $m>1$, so that 
 \begin{equation}
\label{eq:COxHilb}
\kk[\mathscr{M}_0]^{\SL}\cong \bigoplus_{m\geq 0} A^mt^m,
 \end{equation}
 where we adopt the convention that $(A^1)^0=A^0$.

Combining Ginzburg's Lemma~\ref{lem:Ginzburg} with our earlier results leads to the following result:

\begin{theorem}
\label{thm:CoxHilbert}
The Cox ring of $\Hilb^{[n]}(\mathbb{A}^2)$ is the $\mathbb{Z}$-graded ring
\begin{equation}\label{eq:coxHilb}
(\kk[\mathbb{A}^{2n}][t^{-1}]\big)^{S_n} \oplus \bigoplus_{m > 0} \big(\kk[\mathbb{A}^{2n}]_{\sign} t\big)^m,
\end{equation}
where $\kk[\mathbb{A}^{2n}]$ is in degree zero and $\deg t = 1$.
\end{theorem}

\begin{proof}
If we set $V := \mathscr{M}_0$ then, under the isomorphism \eqref{eq:COxHilb}, Proposition~\ref{prop:sect2} says that 
\[
    \Gamma[\mathscr{M}_0^{\theta}]^{\SL} = \bigcap_{f \in \bigcup_{m > 0} (A^m t^m) \setminus \{ 0 \} } \left( \bigoplus_{r \ge 0} A^r t^r \right)[f^{-1}].
\]

We claim that $\kk[\mathscr{M}_0^{\theta}]^{\SL} = \kk[\mathscr{M}_0]^{\SL}[t^{-2}]$. Let $f \in \kk[\mathscr{M}_0^{\theta}]^{\SL}$ be a non-zero homogeneous element of degree $k > 0$. Then $f = p t^k$, where $p \in A^k$. In particular, $p^2 \in A^{2k} \subset A^{2k-2}$. Hence, 
 \[
 t^{-2} = (p^2 t^{2k - 2}) f^{-2}
 \]
 belongs to $\kk[\mathscr{M}_0]^{\SL}[f^{-1}]$, and thus $\kk[\mathscr{M}_0]^{\SL}[t^{-2}] \subset \kk[\mathscr{M}_0^{\theta}]^{\SL}$.
 
We show the opposite inclusion. Let $m \in \kk[\mathscr{M}_0^{\theta}]^{\SL}$ be homogeneous. Since $\kk[\mathscr{M}_0^{\theta}]^{\SL}$ is a subring of $\kk(\mathbb{A}^{2n})[ \GGm]$, where $\kk(\mathbb{A}^{2n})$ is the field of rational functions on $\mathbb{A}^{2n}$, we may write $m = \frac{a}{b} t^k$ uniquely, where $a,b \in \kk[\mathbb{A}^{2n}]$ have highest common factor one and $k \in \ZZ$. If $b$ is not a unit then we choose some irreducible polynomial $f_0 \in \kk[\mathbb{A}^{2n}]$ such that the highest common factor of $\sigma(f_0)$ and $b$ is one, for all $\sigma \in S_n$. Set $f = \prod_{\sigma} \sigma(f_0)$. Then $f \in A^0$ and shares no factors with $b$. Since $f^2 t^2 \in A^2 t^2$, the element $m$ must belong to $\kk[\mathscr{M}_0]^{\SL}[(f^2 t^2)^{-1}]$ and we can write $m = g t^r (f^2 t^2)^{-s}$. But the equality 
$$
\frac{a}{b} t^k = \frac{g t^r}{(f^2 t^2)^{s}}
%\quad \textrm{in }$\mathrm{Frac} \, (P[t])$}
$$
in $\kk(\mathbb{A}^{2n})[ \GGm]$ contradicts uniqueness of factorisation because $b$ is not a unit. We deduce that $\kk[\mathscr{M}_0^{\theta}]^{\SL}$ is contained in $\kk[\mathscr{M}_0]^{\SL}[t^{-2}]$, proving the claim.  

As shown in \cite[Proposition~4.2]{BCS23}, work of McGerty--Nevins~\cite{QuiverKirwan} implies that for any $\theta > 0$ the integral version $\ell_C$ of the linearisation map is an isomorphism. Therefore, we may apply Theorem~\ref{thm:Coxringisom} to deduce that the Cox ring of $\Hilb^{[n]}(\mathbb{A}^2)$ is isomorphic to $\kk[\mathscr{M}_0^{\theta}]^{\SL} = \kk[\mathscr{M}_0]^{\SL}[t^{-2}]$. It remains to show therefore that $\kk[\mathscr{M}_0]^{\SL}[t^{-2}]$ is isomorphic to the graded ring \eqref{eq:coxHilb}. For this, let $m \in \kk[\mathscr{M}_0]^{\SL}[t^{-2}]$. By \eqref{eq:COxHilb}, $m$ is a linear combination of elements of the form $a t^{-2k}$ for $a \in A^r t^r$ and $k \ge 0$. Since $A^r \supset A^{r+2} \supset A^{r+4} \supset \cdots$ we see that $a t^{-2k}$ belongs to $A^{r - 2k} t^{r - 2k} \subset \kk[\mathscr{M}_0]^{\SL}$ if $r \ge 2k$. If $r - 2k < 0$ then $a t^{-2k} \in  (\kk[\mathbb{A}^{2n}][t^{-1}])^{S_n}$. Conversely, if $b t^{-\ell} \in (\kk[\mathbb{A}^{2n}][t^{-1}])^{S_n}$ then either $\ell$ is even so $b t^{-\ell} \in \kk[\mathbb{A}^{2n}]^{S_n}[t^{-2}] \subset \kk[\mathscr{M}_0]^{\SL}[t^{-2}]$, or $\ell$ is odd and 
\[
b t^{-\ell} = (b t) t^{-\ell-1} \in (A^1 t) [t^{-2}] \subset \kk[\mathscr{M}_0]^{\SL}[t^{-2}]
\]
as required. 
\end{proof}

\begin{remark}
 Almost by definition, the Cox ring of a Mori Dream Space $X$ allows one to recover $X$ as a GIT quotient of $\Spec \Cox(X)$ by the torus $\Hom(\Pic(X), \GGm)$, 
 where the stability parameter in $\Pic(X)$ is ample. 
 %\theta$ is chosen to lie in the ample cone of $\Pic(X)$. 
 When $X = \Hilb^{[n]}(\mathbb{A}^2)$, 
 Theorem~\ref{thm:CoxHilbert} implies that $\Hilb^{[n]}(\mathbb{A}^2)$ is Proj of the graded subring of \eqref{eq:coxHilb} generated by all homogeneous elements of non-negative degree, that is, 
    \[
    \Hilb^{[n]}(\mathbb{A}^2) \cong \mathrm{Proj} \, \bigoplus_{m \ge 0} A^m t^m.
    \]
    This recovers the Proj description of $\Hilb^{[n]}(\mathbb{A}^2)$ given by Haiman~\cite[Proposition~2.6]{HaimanCatalan}. 
\end{remark}

\begin{corollary}
\label{cor:linebundleglobal}
 For $m \ge 0$, we have $\Gamma\big(\!\Hilb^{[n]}(\mathbb{A}^2),\mathcal{O}(m)\big) \cong A^m$, and for $m < 0$, we have  \[
    \Gamma\big(\!\Hilb^{[n]}(\mathbb{A}^2),\mathcal{O}(m)\big) = \left\{ \begin{array}{cl}
    A^0 & \textrm{$m$ even} \\
    A^1 & \textrm{$m$ odd}. 
    \end{array} \right.
    \]
\end{corollary}

\begin{remark}
The description %identification 
of $\Gamma(\Hilb^{[n]}(\mathbb{A}^2),\mathcal{O}(m))$
%= A^m$ 
for $m \ge 0$ can also be deduced from deep results of Haiman~\cite{HaimanJAMS,HaimanVanishing}, as shown in \cite[Lemma~4.6]{GordonStaffordI}. 
\end{remark}

\begin{corollary}
    For all $m \ge 0$, we have that
     \[
    (A^m)^{\vee} = \left\{ \begin{array}{ll}
    A^0 & \textrm{$m$ even} \\
    A^1 & \textrm{$m$ odd}. 
    \end{array} \right.
    \]
    In particular, $A^m$ is reflexive if and only if $m = 0,1$. 
\end{corollary}

\begin{proof}
 This result follows from a standard argument based on Grothendieck duality, together with the vanishing theorem of Haiman~\cite{HaimanVanishing}. Indeed, write 
 \[
 \sigma\colon X:=\Hilb^{[n]}(\mathbb{A}^2)\longrightarrow Y:=\mathrm{Sym}^n(\mathbb{A}^2)
 \]
 for the Hilbert--Chow morphism. Note that $Y$ is Gorenstein, $X$ is non-singular, and $\sigma$ is crepant, so the dualising complexes on $X$ and $Y$ are given by $\mathcal{O}_X[2n]$ and $\mathcal{O}_Y[2n]$ respectively.
Therefore, the Grothendieck duality functors on $X$ and $Y$ are %given by 
$D_X(-) = \RR \Hom_X( -,\mathcal{O}_X)[2n]$ and $D_Y(-) = \RR \Hom_Y( -,\mathcal{O}_Y)[2n]$ respectively. Since Grothendieck duality commutes with pushforward along the proper morphism $\sigma$, this implies that $\RR \sigma_* \circ D_X = D_Y \circ \RR \sigma_*$.
 
Haiman's vanishing theorem \cite[Theorem~2.1]{HaimanVanishing} implies that $\RR^i \sigma_* \mathcal{O}(d) = 0$ for $d \ge 0$ and $i > 0$ because $\mathcal{O}(d)$ is a direct summand of the sheaf $P \otimes B^{\otimes l}$ of \textit{loc.\ cit.\ }for $l \gg 0$. Hence the Leray spectral
sequence associated to $D_Y \circ \RR \sigma_* \mathcal{O}(d)$ collapses on the second page. We deduce that 
\begin{equation}\label{eq:dminueglobal}
    \RR^i \sigma_* \mathcal{O}_X(-d) \cong \mathrm{Ext}_{A^0}^{i}\Big(\Gamma\big(X,\mathcal{O}(d)\big),A^0\Big).  
\end{equation}
Taking $i = 0$ in \eqref{eq:dminueglobal} and using Corollary~\ref{cor:linebundleglobal} establishes the corollary.
\end{proof}

\bibliographystyle{plain}

\def\cprime{$'$} \def\cprime{$'$} \def\cprime{$'$} \def\cprime{$'$}
  \def\cprime{$'$} \def\cprime{$'$} \def\cprime{$'$} \def\cprime{$'$}
  \def\cprime{$'$} \def\cprime{$'$} \def\cprime{$'$} \def\cprime{$'$}
  \def\cprime{$'$} \def\cprime{$'$}

 \end{document}